\documentclass[USenglish]{article}	
\usepackage[utf8]{inputenc}				%(only for the pdftex engine)
\usepackage[big,online]{dgruyter}	%values: small,big | online,print,work
\usepackage{lmodern} 
\usepackage{microtype}
\usepackage[numbers,square,sort&compress]{natbib}

\usepackage[cmyk]{xcolor}
\newcommand{\meshset}{\mathbb{T}}				%set of all refinements
\newcommand{\mesh}{\mathcal{T}}					%beliebiges mesh
\newcommand{\meshl}{\mathcal{T}_{\ell}}			%Meshes mit Index \ell 
\newcommand{\norm}[1]{\Vert #1 \Vert }

\renewcommand{\div}{\operatorname{div}}
\newcommand{\osc}{\operatorname{osc}}
\newcommand{\mytensor}[1]{ \pmb{#1} }
\newcommand{\R}{\mathbb{R}} % reelle
\newcommand{\N}{\mathbb{N}} % natuerliche

\theoremstyle{dgthm}
\newtheorem{theorem}{Theorem}

\newtheorem{lemma}{Lemma}

\theoremstyle{dgdef}

\newtheorem{remark}{Remark}

\begin{document}

%%%--------------------------------------------%%%
	\articletype{Research Article}
	\received{Month	DD, 20YY}
	\revised{Month	DD, 20YY}
  \accepted{Month	DD, 20YY}
  \journalname{Computational Methods in Applied Mathematics}
  \journalyear{20YY}
  \journalvolume{XX}
  \journalissue{X}
  \startpage{1}
  \aop
  \DOI{10.1515/sample-YYYY-XXXX}
%%%--------------------------------------------%%%

\title{On an optimal AFEM for elastoplasticity}
\runningtitle{On an optimal AFEM for elastoplasticity}
%\subtitle{Insert subtitle if needed}

\author*[2]{Miriam Schönauer} 
\author[1]{Andreas Schröder}
%\ use * to mark the author as the corresponding author

%\runningauthor{F.~Author et al.}
\affil[1]{\protect\raggedright 
Paris Lodron University of Salzburg, Department of Mathematics, Salzburg, Austria, e-mail: andreas.schroeder@plus.ac.at}
\affil[2]{\protect\raggedright 
Paris Lodron University of Salzburg, Department of Mathematics, Salzburg , Austria, e-mail: miriam.schoenauer@plus.ac.at}
	
%\communicated{...}
%\dedication{...}
	
\abstract{In this paper, optimal convergence for an adaptive finite element algorithm for elastoplasticity is considered. To this end, the proposed adaptive algorithm is established within the abstract framework of the axioms of adaptivity [Comput. Math. Appl., 67(6) (2014), 1195-1253], which provides a specific proceeding to prove the optimal convergence of the scheme. The proceeding is based on verifying four axioms, which ensure the optimal convergence. The verification is done by using results from [Numer. Math., 132(1) (2016), 131-154], which presents an alternative approach to optimality without explicitly relying on the axioms.}

\maketitle

\section{Introduction} 

The behaviour of elastoplastic material under the influence of stress is an important topic in mechanical engineering, see e.g. \cite{ChenElasto, GeoElasto}. 
Elastoplastic materials behave elastically until a certain threshold of stress is exceeded, and then undergo plastic deformation. 
A model problem of elastoplasticity results from  the so-called "primal problem of elastoplasticity with isotropic and linear kinematic
hardening" and can be described through a variational inequality of the second kind \cite{CarstensenDiscrete,CarstensenReliability,SchroederWiedemann}.\\
The use of adaptive finite element methods to approximately solve problems of mechanical engineering has been of high interest,
as it promises an efficient solution with as little computation effort as possible in terms of degree of freedom and/or of computational time. The convergence of adaptive finite element methods was discussed early on in \cite{Bubuska} and proven in \cite{Doerfler} and has been further studied in e.g. \cite{BDVNVB,DorflerWilde, Morin,DataoscConvergence}. Optimal convergence of adaptive finite element methods has been verified in several cases, e.g. \cite{NochettoQuasiOptimal, FeischlOptimal,StevensonStokes,StevensonOptimality}. In \cite{Axioms} an abstract framework for the verification of optimal convergence is outlined. The results of \cite{Axioms} show for the standard loop of

\begin{center}{\textbf{Solve} \(\rightarrow\) \textbf{Estimate} \(\rightarrow\) \textbf{Mark} \(\rightarrow\) \textbf{Refine}}\end{center}

\noindent that if an adaptive finite element method satisfies a set of four axioms, optimal convergence of that method
is guaranteed. These four axioms of adaptivity are "stability" (A1), "reduction property" (A2), "general quasi-orthogonality" (A3) and "discrete reliability" (A4).\\
Adaptive finite element methods are applied to problems of elastoplasticity and, in particular, to the primal problem of elastoplasticity e.g. \cite{CarstensenReliability,CarstensenAveraging,StarkeElasto}. However, results concerning convergence or even optimal convergence are very rare in the literature. In \cite{ElastoAFEM} an adaptive finite element method for elastoplasticity is introduced, and its optimal convergence is proven. However, the proof does not explicitly use the framework of the axioms of adaptivity, but gives the impression that the axioms are applicable to this specific problem.\medskip\\
In this paper, we prove that the adaptive finite element method for elastoplasticity described in \cite{ElastoAFEM} satisfies the axioms of adaptivity of \cite{Axioms}, which guarantee optimal convergence.\\
To that end, we verify that the error measure and error estimation satisfy the necessary requirements. We also recall that the refinement strategy of newest vertex bisection - generalized to dimensions equal to or
higher than two as in \cite{Maubach, StevensonNVB, Traxler} - fulfil the required properties.\\
We apply several results from \cite{ElastoAFEM} to show that the axioms of adaptivity hold. In particular, we use results on properties of the error estimator on refined and non-refined elements of the mesh, as well as a result on "discrete reliability" of the error estimator. It is of significant importance that the error of energies is equivalent to the error measure, which is shown in \cite{ElastoAFEM}.\\
Verifying the axioms, we observe similarities and differences that become apparent between the two proof methodologies \cite{Axioms} and \cite{ElastoAFEM}. For instance, certain results of \cite{ElastoAFEM} act as counterparts to the "stability" axiom (A1), "reduction property" axiom (A2) and "discrete reliability" axiom (A4). Additionally, reliability is treated differently in both papers: as an independent result in \cite{ElastoAFEM} and as a consequence of the "discrete reliability" axiom (A4) in \cite{Axioms}. Both proof methodologies follow the same structure, however, one key difference should be emphasized: the use of efficiency in \cite{ElastoAFEM}. The proof of optimality presented in this paper does not need the efficiency of the error estimator at any point (as it is based on the abstract framework in \cite{Axioms}), whereas in \cite{ElastoAFEM} it is a critical component in the proof of optimal convergence. This shows a clear advantage of the approach of \cite{Axioms}, which remains applicable even if the error estimator is not efficient, whereas the method of \cite{ElastoAFEM} depends on the efficiency of the estimator. This and the streamlined optimality proof provided by the abstract framework of \cite{Axioms} emphasizes the broad applicability of the methods of \cite{Axioms} even to non-linear problems.\\
The paper is organized as follows. In Section~2 we introduce a model problem resulting from one quasi-static time step of the primal problem of elastoplasticity with isotropic and linear kinematic hardening and derive its weak formulation as a variational inequality of second kind. In Section~3 we discuss the discretization with adaptive finite elements. Section~4 is devoted to the proof of optimal convergence using the axioms of adaptivity. We first introduce some preliminaries of \cite{ElastoAFEM} and then show that the axioms (A1)-(A4) hold, which gives the optimal convergence of the scheme. In the last section, we give some remarks on the two proof methodologies, described in \cite{Axioms} and \cite{ElastoAFEM}. This paper uses standard notation of Lebesgue and Sobolev spaces and denotes the \(L^2\) inner product by
\((\cdot,\cdot)_{L^2}\). Throughout this paper, \(X\lesssim Y\) abbreviates \(X\leq C\:Y\) for some positive constant $C$ (which is independent of $Y$).

\section{Model of elastoplasticity}\label{Section2 Model of Elasto}

In this section, we introduce a model problem of elastoplasticity which results from the formulation of one quasi-static time step of the so-called "primal problem of elastoplasticity with isotropic and linear kinematic hardening", see, e.g., \cite{CarstensenDiscrete,ElastoAFEM, HanReddy, SchroederWiedemann}. 
For this purpose, let \(\Omega\subset\R^d\) with \(d\in\{2,3\}\) be the reference domain describing the region of an elastoplastic body without deformation. We assume that \(\Omega\) has a Lipschitz boundary \(\partial\Omega\). The boundary is separated into the closed and non-empty Dirichlet part \(\Gamma_D\subset\partial\Omega\) with positive surface measure and the open Neumann part \(\Gamma_N:=\partial\Omega\backslash\Gamma_D\), which has the outer unit normal \(\nu\) and may be empty. Two forces act on \(\Omega\) namely the body force \(f\in L^2(\Omega;\R^d)\) and the surface traction \(g\in L^2(\Gamma_N;\R^d)\). 
We denote the field of displacements by 

\begin{equation*}
V:=\{v\in H^1(\Omega;\R^d) \mid v=0 \text{ on }\Gamma_D\}
\end{equation*}

\noindent and the linearized Green tensor $\mytensor{\varepsilon}$ by

\begin{equation*}
\mytensor{\varepsilon}(w):=\frac12\: \big(\nabla w+(\nabla w)^T\big)
\end{equation*}

\noindent for $w\in V$.

%---------------------------
% Equilibrium Equation
%---------------------------

\subsection{Equilibrium equation}

In the framework of elastoplasticity with small strain, the tensor is usually decomposed as

\begin{equation*}
\mytensor{\varepsilon}(w)=\mytensor{e}+\mytensor{p},
\end{equation*}

\noindent where \(\mytensor{e}\) is the elastic part and \(\mytensor{p}\) the plastic part. By \(Q\) we denote the space of plastic strains, i.e. 

\begin{equation*}
Q:=L^2\left(\Omega;\mathbb{S}_{d,0}\right)
\end{equation*}

\noindent with 

\begin{equation*}
\mathbb{S}_{d,0}=\left\{\mytensor{q}\in\mathbb{S}_d\mid \sum\limits_{j=1}^d q_{jj}=0\right\},\quad \mathbb{S}_d:=\{\mytensor{q}\in\R^{d\times d}\mid  q_{ji}=q_{ij}\}
\end{equation*}

\noindent and introduce the stress tensor 

\begin{equation*}
\mytensor{\sigma}(w,\mytensor{p}):=\mathbb{C}( \mytensor{\varepsilon}(w)-\mytensor{p})
\end{equation*}

\noindent with \(\mathbb{C}\) denoting the isotropic elasticity tensor, which is assumed to be symmetric, constant and to fulfil

\begin{equation*}
\kappa_{\mathbb{C}}\:|\mytensor{\tau}|^2\leq \mathbb{C}\mytensor{\tau} : \mytensor{\tau}
\end{equation*}

\noindent for all $\mytensor{\tau}\in L^2(\Omega;\mathbb{S}_d)$ with the constant $\kappa_{\mathbb{C}}>0$.\\
With these notations, we formulate the equilibrium equations by
\begin{subequations}
    \begin{align}
        \operatorname{div}\mytensor{\sigma}(w,\mytensor{p})+f &= 0\text{ in }\Omega\label{equilibrium}\\
        \mytensor{\sigma}(w,\mytensor{p})\nu &= g\text{ on } \Gamma_N \label{NeumanBoundary}
    \end{align}
\end{subequations}

\noindent for some $(w,\mytensor{p})\in V\times Q$. Multiplying \eqref{equilibrium} with a test function $v\in V$ and applying integration by parts, we conclude that $(w,\mytensor{p})$ fulfils \eqref{equilibrium} and \eqref{NeumanBoundary} if and only if 

\begin{equation}
(\mytensor{\sigma}(w,\mytensor{p}),\mytensor{\varepsilon}(v))_{L^2(\Omega;\R^{d\times d})}=(f,v)_{L^2(\Omega;\R^d)}+(g,v)_{L^2(\Gamma_N;\R^d)}
\label{equilibrium transform}
\end{equation}

\noindent for all \(v\in V\).

%----------------------------
% Linear Hardening
%----------------------------

\subsection{Linear hardening}
The linear hardening state law is given by

\begin{equation}
\chi = \mathbb{H}\mytensor{p},\quad R = H\alpha,
\label{stateLaw}
\end{equation}

\noindent where \(H>0\) denotes the isotropic hardening modulus and \(\mathbb{H}\) is the hardening tensor, which is assumed to be symmetric and to satisfy

\begin{equation*}
\kappa_{\mathbb{H}}\:|\mytensor{\tau}|^2\leq (\mathbb{H}\mytensor{\tau}) : \mytensor{\tau} 
\end{equation*}

\noindent for all $\mytensor{\tau}\in L^2(\Omega;\mathbb{S}_d$) with the constant \(\kappa_{\mathbb{H}}>0\). In \eqref{stateLaw}
\(\chi\in Q\) is the back stress tensor and \(\alpha\in M:=L^2(\Omega)\) is the accumulated plastic strain with its dual variable \(R\in M\).\\
From the time discrete version of the elastoplastic evolution laws, it holds 

\begin{equation}\label{evolutionlaw}
(\mytensor{\sigma},\chi,R)\in \partial j(\mytensor{p},\alpha)
\end{equation}

\noindent where \(\partial j \) is the subdifferential of the dissipation potential \(j\) \cite{FunConvex}, which is assumed to be convex, lower semi-continuous and positively homogeneous of degree one. In the isotropic hardening case, the dissipation potential is given by

\begin{equation}\label{convex functional}
j(\mytensor{q},\beta):=\begin{cases}
\sigma_y|\mytensor{q}|, &|\mytensor{q}|\leq \beta,\\
\infty, &\text{ else}
\end{cases}
\end{equation}

\noindent with the constant yield stress \(\sigma_y>0\). The expression \eqref{evolutionlaw} can equivalently be written as

\begin{align}
(\mathbb{H}\mytensor{p}-\mytensor{\sigma}(w,\mytensor{p}),\mytensor{p}-\mytensor{q})_{L^2(\Omega;\R^{d\times d})}+(H\alpha,\alpha-\beta)_{L^2(\Omega)}+\int\limits_{\Omega}j(\mytensor{p},\alpha)-j(\mytensor{q},\beta)\,dx\leq 0
\label{subdif dual}
\end{align}

\noindent for all \((\mytensor{q},\beta)\in Q\times M\).

%-----------------------------
% Weak formulation
%-----------------------------

\subsection{Weak formulation of variational inequality}

Adding \eqref{equilibrium transform} to \eqref{subdif dual},  we find that $(w,\mytensor{p})\in V\times Q$ and $\alpha\in M$ fulfil \eqref{equilibrium transform} for all $v\in V$ and \eqref{subdif dual} for all $(\mytensor{q},\beta)\in Q\times M$, respectively, if and only if $u:=(w,\mytensor{p},\alpha)\in X:=V\times Q\times M$ 
satisfies the variational inequality

\begin{equation}
b(z-u)\leq a(u,z-u)+\psi(z)-\psi(u)
\label{weak form}
\end{equation}

\noindent for all \(z\in X\), where 

\begin{align*}
a(u,z)&:=(\mytensor{\sigma}(w,\mytensor{p}),\mytensor{\varepsilon}(v)-\mytensor{q})_{L^2(\Omega;\R^{d\times d})}+(\mathbb{H}\mytensor{p},\mytensor{q})_{L^2(\Omega;\R^{d\times d})}+(H\alpha,\beta)_{L^2(\Omega)},\\
\psi(z)&:=\int\limits_{\Omega}j(\mytensor{q},\beta)\,dx ,\\
b(z)&:=(f,v)_{L^2(\Omega;\R^d)}+(g,v)_{L^2(\Gamma_N;\R^d)}
\end{align*}

\noindent for \(u=(w,\mytensor{p},\alpha)\) and \(z=(v,\mytensor{q},\beta)\in X\). Equipping the space \(X\) with the inner product

\begin{equation*}
(u,z)_X:=(w,v)_{H^1(\Omega;\R^d)}+(\mytensor{p},\mytensor{q})_{L^2(\Omega;\R^{d\times d})}+(\alpha,\beta)_{L^2(\Omega)}
\end{equation*}

\noindent and using the norm \(\norm{u}_X:=(u,u)_X^{1/2}\), we conclude that
\(b(\cdot)\) is continuous, that \(\psi(\cdot)\) is a convex, lower semi-continuous and positive homogeneous functional and that \(a(\cdot,\cdot)\) is symmetric and continuous. Moreover, Korn’s inequality yields that \(a(\cdot,\cdot)\) is $X$-coercive, i.e. it satisfies

\begin{equation*}
\kappa\: \norm{z}^2_X\leq a(z,z)
\end{equation*}

\noindent for all \(z\in X\) with a constant \(\kappa>0\) \cite{HanReddy}. Therefore, there exists a unique solution of \eqref{weak form}, which is also the unique minimizer of the functional

\begin{equation*}
E(z):=\frac12\: a(z,z)-b(z)+\psi(z)
\end{equation*}

\noindent with $z\in X$, see \cite{HanReddy}.

%--------------------------------------
%Discretization with AFE
%--------------------------------------
\section{Discretization with adaptive finite elements}\label{Section3 Discretization}

Let \(\mesh\) be a mesh, i.e.~a finite decomposition of \(\Omega\) such that

\begin{equation*}
\overline{\Omega}=\bigcup\limits_{T\in \mesh}T
\end{equation*}

\noindent where the elements of $\mesh$ are closed triangles if $d=2$ and closed tetrahedrons if $d=3$.
As usual, we assume that $\mesh$ is conforming, i.e. for two elements $T,T’\in\mesh$ it holds that $T\cap T’$ is either empty or a face, an edge or a vertex of $T$ and $T'$.\\
In the following subsection, we briefly list the essential components of the discretization with adaptive finite elements to be used in Section \ref{Section4 An optimal AFEM}.

%--------------------------
% Discrete Solution
%--------------------------

\subsection{Discrete solution}\label{discreteSolution}

Let \(\mathcal{P}_k(\mesh;\mathbb{R}^d)\) be the space of (vector-valued) piecewise affine linear functions for $k=1$ or the space of (vector-valued) piecewise constant functions for $k=0$ with respect to \(\mesh\). We set

\begin{equation*}
V(\mesh):=\mathcal{P}_1(\mesh;\mathbb{R}^d)\cap V,\quad Q(\mesh):=\mathcal{P}_0(\mesh;\mathbb{S}_{d,0}),\quad M(\mesh):=\mathcal{P}_0(\mesh)
\end{equation*}

\noindent and define \( X(\mesh):=V(\mesh)\times Q(\mesh)\times M(\mesh) \). The discrete problem of \eqref{weak form} consists in finding a discrete solution \(U(\mesh)\in X(\mesh)\), which satisfies

\begin{equation}
b(z-U(\mesh))\leq a(U(\mesh),z-U(\mesh))+\psi(z)-\psi(U(\mesh))
\label{Discrete Variational inequality}
\end{equation}

\noindent for all  $z\in X(\mesh)$.
The existence and uniqueness of $U(\mesh)$ follows from the same arguments as in the non-discretized case  \eqref{weak form} and minimizes the operator \(E\) over \(X(\mesh)\), see \cite{HanReddy}.

%-----------------------
%Error Measure
%-----------------------

\subsection{Error measure}\label{errorMeasure}

We introduce the error measure $\mathrm{d}[\cdot,\cdot]$ as

\begin{equation*}
\mathrm{d}[z,\hat{z}]:=\mathrm{d}[z,\hat{z};\Omega]:= \norm{\mytensor{\sigma}(v,\mytensor{q})-\mytensor{\sigma}(\hat{v},\hat{\mytensor{q}})}_{L^2(\Omega;\R^{d\times d})}
\end{equation*}

\noindent with $z=(v,\mytensor{q},\beta),\; \hat{z}=(\hat{v},\hat{\mytensor{q}},\hat{\beta}) \in X$. It holds $\mathrm{d}[z,\hat{z}]\geq 0$ (non-negativity), $\mathrm{d}[z,\hat{z}]\geq C_\Delta \mathrm{d}[\hat{z},z]$ (quasi-symmetry) and $C_\Delta^{-1} \mathrm{d}[z,\hat{z}]\leq \mathrm{d}[\hat{z},y]+d[y,\hat{z}]$ (quasi-triangle inequality) for all $z,\hat{z},y\in  X \cup X(\mesh)$ where $C_\Delta:=1$. Moreover, $\mathrm{d}[z,\hat{z}]$ is independent of $\mesh$ which means that the so-called "compatibility condition" \cite[Section 2.2]{Axioms} is satisfied. Finally, by using results from \cite{HanReddy} the so-called "further approximation property" \cite[Section 2.2]{Axioms} is also fulfilled. These properties have to be fulfilled in the abstract framework of \cite{Axioms}.

%---------------------------------------------
%Refinement indicator & Global error estimator
%---------------------------------------------
\subsection{Refinement indicator and global error estimator}\label{refinementIndicator}

We define the refinement indicator \(\eta_T(\mesh;\cdot)\) as

\begin{equation*}
\eta_T(\mesh;z):= |T|\:\Vert f\Vert_{L^2(T;\mathbb{R}^d)}^2+|T|^{1/2}\sum\limits_{E\in\mathcal{E}(T)}R_E^2(z)
\end{equation*}

\noindent for $T\in\mathcal{T}$ and $z=(v,\mytensor{q},\beta
)\in X(\mesh)$, where

\begin{equation*}
R_E(z):=\begin{cases}
\norm{[\mytensor{\sigma}(v,\mytensor{q})]\nu_E}_{L^2(E;\R^d)}& E\in \mathcal{E}(\mesh),\\
\norm{g-\mytensor{\sigma}(v,\mytensor{q})\nu_E}_{L^2(E;\R^d)}& E\in \mathcal{E}_N(\mesh),\\
0& E\in \mathcal{E}_D(\mesh).
\end{cases}
\end{equation*}

\noindent Here, \([\cdot]\) denotes the jump along an edge (or face) $E$ of \(\mesh\), \(\nu_E\) represents the fixed unit normal vector corresponding to $E$ and $|T|$ is the measure of $T\in\mathcal{T}$. Moreover, \(\mathcal{E}(T)\) is the set of all edges (or faces) of $T$, \(\mathcal{E}(\mesh)\) is the set of all interior edges (or faces), \(\mathcal{E}_N(\mesh)\) is the set of all edges (or faces) on \(\Gamma_N\) and \(\mathcal{E}_D(\mesh)\) is the set of all edges (or faces) on \(\Gamma_D\).
The global error estimator \(\eta(\mesh;\cdot)\) is given by

\begin{equation*}
\eta(\mesh;z)^2:=\sum\limits_{T\in\mesh}\eta_T(\mesh;z)^2
\end{equation*}

\noindent for $z\in X(\mesh)$. We note that the error estimator is reliable, i.e.

\begin{equation}
\mathrm{d}[u,U(\mathcal{T})]\lesssim \eta(\mathcal{T};U(\mathcal{T})).
\label{reliability}
\end{equation}

\noindent The proof of the reliability can be found in \cite{CarstensenReliability}.
The error estimator is also efficient, i.e.~
\begin{equation}
    \eta(\mesh;U(\mesh))^2 \lesssim \mathrm{d}[u,U(\mesh)]^2+\osc^2(\mesh)
    \label{efficiency}
\end{equation}
\noindent with the oscillations $\osc^2(\mesh):=\operatorname{osc}^2(f,\mesh)+\operatorname{osc}^2(g,{\cal E}_N(\mesh))$ given by

\begin{align*}
\osc^2(T;f)&:=|T|\: \|f-f_T\|_{L^2(T;{\mathbb R}^d)}^2,\quad \operatorname{osc}^2(f,\mathcal{ M}):=\sum_{T\in {\mathcal M}}\osc^2(T;f),\\
 \osc^2(E;g)&:=|T|^{1/2}\: \|g-g_E\|_{L^2(E;{\mathbb R}^d)}^2,\quad \operatorname{osc}^2(g,{\cal F}):=\sum_{E\in {\mathcal F}} \osc^2(E;g)
\end{align*}

\noindent for $\mathcal{M}\subset \mesh \text{ and } {\cal F}\subset {\cal E}_N(\mesh)$, 
where the integral means are defined as
$f_T:=|T|^{-1}\int_T f\,dx$ and $g_E:=|E|^{-1}\int_E g\,ds$ for $E\in {\cal E}_\ell(\Gamma_N)$. Again, we refer to \cite{CarstensenReliability} for a proof.

%-------------------------------
% Newest Vertex Bisection
%-------------------------------

\subsection{Newest vertex bisection}\label{NVB}

We briefly recall the mesh refinement with the newest vertex bisection, where we use the same notation as in \cite{PraetoriusGantner}. For this purpose, let \(\mathcal{M}\subset \mesh\) indicate those elements of $\mesh$ which are indicated to be refined and let \(E_T\) be a fixed edge of $T\in\mesh$ (the reference edge). The newest vertex bisection generates a new mesh $\operatorname{NVB}(\mesh,\mathcal{M})$ in the following way: First, each \(T\in\mathcal{M}\) is bisected into two triangles or tetrahedrons, respectively, where the midpoint of $E_T$ (the newest vertex) is connected to the vertex opposite to $E_T$. The edges opposite to the newest vertex are the reference edges of the new triangles or tetrahedrons.  
Second, in the same way all elements $T$ of the resulting mesh which have a hanging node are bisected (i.e. a vertex of the resulting mesh which is in the interior of an edge of $T$). The second refinement step is performed until the resulting mesh has no hanging nodes.
Note that the second step (the mesh-closure step) leads to a finite number of additional bisections.   
We refer to \cite{BDVNVB,StevensonNVB} for more details on the newest vertex bisection.
By $\mathbb{T}=\mathbb{T}(\mesh)$ we denote the set of all meshes which are created by a finite number of successive applications of the newest vertex bisections of $\mesh$. Note that all meshes $\hat{\mesh}\in\mathbb{T}$ fulfil 

\begin{equation}
|\mesh\backslash\hat{\mesh}|\leq |\hat{\mesh}|-|\mesh|
\label{ersteMeshProp}
\end{equation}

\noindent and are uniformly shape regular, i.e. there exists a constant $c>0$ such that $h_T\leq c\: \rho_T$ for all $T\in \hat{\mesh}$ and all $\hat{\mesh}\in\mathbb{T}$, 
where $h_T$ is the maximum diameter and $\rho_T$ is  the maximum inner circle radius  of \(T\), see \cite{BDVNVB,StevensonNVB}.
Further note that

\begin{equation}
|\operatorname{NVB}(\mesh,\mathcal{M})|-|\mesh|\leq (C_{\text{son}}-1)\:|\mesh|
\label{max refinement anzahl}
\end{equation}

\noindent with a constant $C_{\text{son}}>0$. 
We refer to \cite[Cor.3.5]{StevensonNVBRemark} which implies that the constant \(C_{\text{son}}\) is finite for $d\geq 2$. It is well-known (see  \cite{NochettoQuasiOptimal}) that there exists a coarsest common refinement \(\mesh'\oplus\mesh''\in\meshset(\mesh')\cap \meshset(\mesh'')\subset\meshset \) of two meshes $\mesh',\mesh''\in\mathbb{T}$, which satisfies

 \begin{equation}
|\mesh'\oplus\mesh''|\leq |\mesh'|+|\mesh''|-|\mesh|.
\label{CCRestimate}
\end{equation}

%------------------------
%Adaptive Refinements
%------------------------

\subsection{Adaptive refinements}\label{adaptiveRefinements}

Adaptive refinements based on the error indicator $\eta(\mesh_{\ell};\cdot)$ are given by a sequence of meshes $\mesh_\ell$ with an initial finite decomposition $\mesh_0$ of $\Omega$ and

\begin{equation*}
    \mesh_{\ell+1}:=\operatorname{NVB}(\mesh_\ell,\mathcal{M}_\ell)
\end{equation*}

\noindent for $\ell\in\mathbb{N}_0$, 
where $\mathcal{M}_\ell\subset \mesh_\ell$ is a set of minimal cardinality such that

\begin{equation*}
\theta\: \eta(\mesh_{\ell};U(\mesh_\ell))^2\leq \sum\limits_{T\in\mathcal{M}_{\ell}}\eta_T(\mesh_{\ell};U(\mesh_{\ell}))^2
\end{equation*}

\noindent for a bulk parameter $0<\theta\leq 1$. It holds 

\begin{equation}
|\meshl|-|\mesh_0|\leq C_{\text{mesh}}\sum\limits_{k=0}^{\ell-1}|\mathcal{M}_k|
\label{BDVestimate}
\end{equation}

\noindent for all $\ell\in\N$ and a constant \(C_{\text{mesh}}>0\) depending on \(\meshset(\mesh_0)\).
For $d=2$ the proof of \eqref{BDVestimate} can be found in \cite{BDVNVB}. We refer to \cite[Thm.6.1]{StevensonNVB} for the case $d=3$.

\begin{remark}
We note that instead of the newest vertex bisection, any refinement strategy may be chosen as long as it satisfies the properties \eqref{ersteMeshProp}-\eqref{BDVestimate}.
\end{remark}

%-------------------------
% Axioms of Adaptivity
%-------------------------
\section{An optimal adaptive finite element method}\label{Section4 An optimal AFEM}

%--------------------------
%Preliminaries
%--------------------------
\subsection{Preliminaries}

In this subsection, we collect some results from \cite{ElastoAFEM} which are needed in the following subsections. They are formulated either for the meshes $\meshl$ generated by the adaptive refinements of Section \ref{adaptiveRefinements} or for some general mesh $\mesh$.   

\noindent Let $\hat{\mesh}\in \mathbb{T}$ and let $\hat{\mathcal{S}}(T)\subset \hat{\mesh}$ be the set of those elements which are generated by some newest vertex bisections of $T\in \mesh$. The result of \cite[Lem.4]{ElastoAFEM} states that there exists \(\Lambda_1>0\) such that

\begin{equation}
\sum_{\hat{T}\in\mathcal{S}(T)} \eta_{\hat{T}}(\hat{\mesh};U(\hat{\mesh}))^2\leq 2^{-1/2}(1+\lambda)\:\eta_T(\mesh;U(\mesh))^2+\Lambda_1(1+1/\lambda)\:\mathrm{d}[U(\hat{\mesh}),U(\mesh);\omega_T]^2
\label{lem4}
\end{equation}

\noindent for all \(\lambda>0\) and all \(T\in\mesh\backslash\hat{\mesh}\).  Here, $\omega_T:=\bigcup_{E\in\mathcal{E}(T)} \omega_E$ and \(\omega_{E}:=\operatorname{int}(T\cup T')\) for \(T,T'\in \mesh\) and \(E\in \mathcal{E}(T)\cap\mathcal{E}(T')\). 
The result of \cite[Thm.6]{ElastoAFEM} yields the existence of \(\beta\geq 0\) and \(0<\rho<1\) such that

\begin{equation}
\eta(\mesh_{\ell+1};U(\mesh_{\ell+1}))^2+\beta\:(E(U(\mesh_{\ell+1}))-E(u))\leq \rho \:\Big(\eta(\meshl;U(\meshl))^2+\beta\big(E(U(\meshl))-E(u)\big)\Big)
\label{etaE}
\end{equation}

\noindent for all $\ell\in \mathbb{N}$. Furthermore, it holds 

\begin{equation}
E(U(\mesh))-E(u)\lesssim \mathrm{d}[u,U(\mesh)]^2\lesssim E(U(\mesh))-E(u)
\label{ESigma}
\end{equation}

\noindent which is \cite[Thm.1]{ElastoAFEM}. The result of \cite[Thm.3]{ElastoAFEM} provides the "discrete reliability"

\begin{equation}
    \mathrm{d}[U(\hat{\mesh}),U(\mesh)]^2\lesssim \sum\limits_{T\in\mesh\backslash\hat{\mesh}}\eta_T(\mesh;U(\mesh))^2
    \label{discreteReliability}
\end{equation}

\noindent for all $\hat{\mesh}\in\mathbb{T}$. Finally, we note some properties of the oscillations defined in Section \ref{refinementIndicator}, namely

\begin{equation}
 \osc^2(f;\hat{S}(T))\leq\osc^2(f;T)\leq \eta_T(\meshl;U(\meshl))^2,
 \label{osc estimate 1}
\end{equation}

\noindent for all $T\in\mesh$ and

\begin{equation}
  \osc^2(\hat{\mathcal{E}}_N(F);g)\leq\osc^2(F;g)\leq \eta_T(\meshl;U(\meshl))^2 
  \label{osc estimate 2}
\end{equation}

\noindent for all $F\in {\cal E}(T)\cap {\cal E}_N(\mesh)$ where
$\hat{\mathcal{E}}_N(F):=\{\hat{F}\in {\cal E}_{N}(\hat{S}(T))\mid \hat{F}\subset F\}$. We refer to \cite[Lem.2]{ElastoAFEM} for a proof.

\begin{remark}
The equivalence of the error of energies and error of stresses stated in \eqref{ESigma} is fundamental to the proof of optimality in \cite{ElastoAFEM}. The right-hand side of this chain of inequalities can be derived from the variational inequality \eqref{weak form}. To gain the inequality on the left-hand side, Jensen's inequality is applied (see \cite{ElstrodtJensen}) to the convex functional $j$ as introduced in \eqref{convex functional}. Then orthogonality properties of the $L^2$-projection are utilized to obtain the desired inequality.
\end{remark}

%--------------------------
% Verifying the Axioms
%--------------------------
\subsection{The axioms of adaptivity}

In this section, we verify the four axioms of adaptivity (A1) - (A4) (see introduction) for the adaptive discretization of the model problem of elastoplasticity, i.e.~for the discrete solution $U$ of Section \ref{discreteSolution}, the error measure $\mathrm{d}[\cdot,\cdot]$ of Section \ref{errorMeasure} and the refinement indicator $\eta_T$ and the global error estimator $\eta$ in conjunction with the newest vertex bisection and the adaptive refinements of the Sections \ref{NVB} and \ref{adaptiveRefinements}.\\
We start with the first axiom (A1), which is "stability on non-refined element domains" i.e. on subsets of $\mesh\cap\hat{\mesh}$ for $\hat{\mesh}\in\meshset$.

%---------------------
% Proof A1
%---------------------
\begin{lemma}[A1] \label{LemA1}
  There exists a constant $C_1>0$ such that
\begin{equation*}
\bigg|\bigg(\sum\limits_{T\in S}\eta_T(\hat{\mesh};\hat{z})^2\bigg)^{1/2}-\bigg(\sum\limits_{T\in S}\eta_T(\mesh;z)^2\bigg)^{1/2}\bigg|\leq C_1\:\mathrm{d}[\hat{z},z]
\end{equation*}

\noindent for all \(\hat{\mesh}\in\mathbb{T}\), \(\mathcal{S}\subseteq \mesh\cap\hat{\mesh}\), $z\in X(\mesh)$ and  $\hat{z}\in X(\hat{\mesh})$.
\end{lemma}

\begin{proof}
Let \(\mathcal{S}:=\{T_1,\dots, T_n\}\) and 
\(E_{i,1},\dots,E_{i,m}\) with $m\in\{3,4\}$ be the edges (or faces) of $T_i$, $i\in\{1,\ldots,n\}$.
Defining the vector $\gamma(z):=(\gamma^1(z),\ldots,\gamma^n(z))\in\R^{n(m+1)}$ with subvectors $\gamma^i(z)\in\mathbb{R}^{m+1}$ defined as

\begin{equation*}
\gamma^i(z):=\left(|T_i|^{1/2}\:\norm{f}_{L^2(T_i;\R^d)},
|T_i|^{1/4}\: R_{E_{i,1}}(z), \ldots,|T_i|^{1/4}\:R_{E_{i,m}}(z)\right)
\end{equation*}

\noindent we write

\begin{equation*}
\left|\left(\sum\limits_{T\in S}\eta_T(\hat{\mesh};\hat{z})^2\right)^{1/2}-\left(\sum\limits_{T\in S}\eta_T(\mesh;z)^2\right)^{1/2}\right|=\big|\norm{\gamma(\hat{z})}-\norm{\gamma(z)}\big|.
\end{equation*}

\noindent Using the reverse triangle inequality, we get

\begin{align*}
\big|\norm{\gamma(\hat{z})}-\norm{\gamma(z)}\big|^2 & \leq \norm{\gamma(\hat{z})-\gamma(z)}^2\\
& =\sum\limits_{i=1}^n |T_i|^{1/2}\sum\limits_{j=1}^m \left(R_{E_{i,j}}(\hat{z})-R_{E_{i,j}}(z)\right)^2\\
& =\sum\limits_{i=1}^n |T_i|^{1/2}\left(\sum\limits_{E\in \mathcal{E}(T_i)\cap\mathcal{E}(\mesh_\ell)} \left(R_{E}(\hat{z})-R_{E}(z)\right)^2+ \sum\limits_{E\in \mathcal{E}(T_i)\cap\mathcal{E}_N(\mesh_\ell)}\left(R_{E}(\hat{z})-R_{E}(z)\right)^2\right)\\
& =\sum\limits_{i=1}^n |T_i|^{1/2}\Bigg(\sum\limits_{E\in \mathcal{E}(T_i)\cap\mathcal{E}(\mesh)}\left(\norm{[\mytensor{\sigma}(\hat{v},\hat{\mytensor{q}})]\nu_{E}}_{L^2(E;\R^d)}-\norm{[\mytensor{\sigma}(v,\mytensor{q})]\nu_{E}}_{L^2(E;\R^d)}\right)^2\\
&\quad + \sum\limits_{E\in\mathcal{E}(T_i)\cap\mathcal{E}_N(\mesh)}\left(\norm{g-\mytensor{\sigma}(\hat{v},\hat{\mytensor{q}})\nu_{E}}_{L^2(E;\R^d)}-\norm{g-\mytensor{\sigma}(v,\mytensor{q})\nu_{E}}_{L^2(E;\R^d)}\right)^2\Bigg).
\end{align*}

\noindent Note that the trace theorem and the shape regularity, see \cite{ElastoAFEM}, imply that 
\begin{equation*}
    \norm{[\mytensor{\sigma}(\hat{v},\hat{\mytensor{q}})-\mytensor{\sigma}(v,\mytensor{q})]\nu_{E}}_{L^2(E;\R^d)}\lesssim |E|^{-1/2}\: \mathrm{d}[\hat{z},z;\omega_E]
\end{equation*}
\noindent for $E\in \mathcal{E}(\mesh)$ and

\begin{equation*}
\norm{(\mytensor{\sigma}(\hat{v},\hat{\mytensor{q}})-\mytensor{\sigma}(v,\mytensor{q}))\nu_{E}}_{L^2(E;\R^d)}\lesssim |E|^{-1/2}\: \mathrm{d}[\hat{z},z;\omega_E]
\end{equation*}

\noindent for $E\in\mathcal{E}_N(\mesh)$. This and reapplying the reverse triangle inequality yield

\begin{align*}
\big|\norm{&\gamma(\hat{z})}-\norm{\gamma(z)}\big|^2\\
\quad & \leq\sum\limits_{i=1}^n |T_i|^{1/2}\left(\sum\limits_{E\in \mathcal{E}(T_i)\cap\mathcal{E}(\mesh_\ell)}\norm{[\mytensor{\sigma}(\hat{v},\hat{\mytensor{q}})-\mytensor{\sigma}(v,\mytensor{q})]\nu_{E}}_{L^2(E;\R^d)}^2\right.
\left.+ \sum\limits_{E\in\mathcal{E}(T_i)\cap\mathcal{E}_N(\mesh_\ell)} \norm{(\mytensor{\sigma}(v,\mytensor{q})-\mytensor{\sigma}(\hat{v},\hat{\mytensor{q}}))\nu_{E}}_{L^2(E;\R^d)}^2\right)\\
&\lesssim \sum\limits_{i=1}^n |T_i|^{1/2}\left(\sum\limits_{E\in \mathcal{E}(T_i)\cap(\mathcal{E}(\mesh)\cup\mathcal{E}_N(\mesh))} |E|^{-1}\: \mathrm{d}[\hat{z},z;\omega_E]^2\right)\\
&= \sum\limits_{i=1}^n \sum\limits_{\substack{j=1}}^m |T_i|^{1/2} |E_{ij}|^{-1}\mathrm{d}[\hat{z},z;\omega_{E_{ij}}]^2\\
&\lesssim \sum\limits_{i=1}^n \sum\limits_{\substack{j=1}}^m h_{T_i} \rho_{T_i}^{-1}\mathrm{d}[\hat{z},z;\omega_{E_{ij}}]^2\\
& \lesssim \mathrm{d}[\hat{z},z]^2.
\end{align*}

\end{proof}

The next axiom we verify is axiom (A2), which is "Reduction property on refined element domains".

%-------------------------
% Proof A2
%-------------------------
\begin{lemma}[A2]\label{LemA2}
There exist constants $0<\rho_2<1$ and $C_2>0$ such that

\begin{equation*}
\sum\limits_{\hat{T}\in\hat{\mesh}\backslash \mesh}\eta_{\hat{T}}(\hat{\mesh};U(\hat{\mesh}))^2\leq \rho_2\sum\limits_{T\in\mesh\backslash \hat{\mesh}}\eta_T(\mesh;U(\mesh))^2+C_2\:\mathrm{d}[U(\hat{\mesh}),U(\mesh)]^2
\end{equation*}

\noindent for all \(\hat{\mesh}\in\mathbb{T}\).
\end{lemma}

\begin{proof}
From \eqref{lem4} we obtain 

\begin{align*}
\sum\limits_{\hat{T}\in\hat{\mesh}\backslash \mesh}\eta_{\hat{T}}(\hat{\mesh};U(\hat{\mesh}))^2
&= \sum_{T\in\mesh} \sum_{\hat{T}\in\mathcal{S}(T)} \eta_{\hat{T}}(\hat{\mesh};U(\hat{\mesh}))^2\\
& \leq   2^{-1/2}(1+\lambda) \sum_{T\in\mesh} \eta_T(\mesh;U(\mesh))^2+\Lambda_1(1+1/\lambda) \sum_{T\in\mesh}\mathrm{d}[U(\hat{\mesh}),U(\mesh);\omega_T]^2.
\end{align*}

\noindent The assertion follows with some $0<\lambda<2^{1/2}-1$.
\end{proof}

Next, we show that the "general quasi-orthogonality" axiom (A3) is fulfilled.

%---------------------
% Proof A3
%---------------------
\begin{lemma}[A3] \label{LemA3}
There exist constants

\begin{equation*}
0\leq \varepsilon_3<\varepsilon_3^*(\theta):=\sup\limits_{\delta>0}\frac{1-(1+\delta)(1-(1-\rho_2)\:\theta)}{C_{\text{rel}}^2\:(C_2+(1+\delta^{-1})\:C_1^2)}
\end{equation*}

\noindent and \(C_3(\varepsilon_3)\geq 1\) such that

\begin{equation*}
\sum\limits_{k=\ell}^N\Big(\mathrm{d}[U(\mesh_{k+1}),U(\mesh_k)]^2-\varepsilon_3\:\mathrm{d}[u,U(\mesh_k)]^2\Big)\leq C_3(\varepsilon_3)\:\eta(\mesh_{\ell};U(\mesh_{\ell}))^2
\end{equation*}

\noindent for all \(\ell,N\in\mathbb{N}_0\) with \(N\geq \ell\).
\end{lemma}

\begin{proof}
As $u$ minimizes $E$, we have $E(U(\mesh_{\ell+k}))-E(u)\geq 0$. Thus, we directly conclude from \eqref{etaE}

\begin{align*}
\eta(\mesh_{\ell+k};U(\mesh_{\ell+k}))^2 & \leq \eta(\mesh_{\ell+k};U(\mesh_{\ell+k}))^2+\beta\:\big(E(U(\mesh_{\ell+k}))-E(u)\big)\\
& \leq \rho^k\: \Big((\eta(\meshl;U(\meshl))^2+\beta\:\big(E(U(\meshl))-E(u)\big)\Big).
\end{align*}

\noindent  Exploiting \eqref{reliability} and \eqref{ESigma} we have

\begin{equation*}
E(U(\meshl))-E(u)\lesssim \mathrm{d}[u,U(\meshl)]^2\lesssim \eta(\meshl;U(\meshl))^2
\end{equation*}

\noindent and, thus,

\begin{equation*}
\eta(\mesh_{\ell+k};U(\mesh_{\ell+k}))^2\lesssim \rho^k\:\eta(\meshl;U(\meshl))^2,
\end{equation*}

\noindent which is the uniform R-linear convergence on any level. Eventually, \cite[Prop.~ 4.11]{Axioms} states that this together with the reliability \eqref{reliability} implies the assertion.
\end{proof}

Finally, we verify the "discrete reliability", which is axiom (A4).

%-------------------------
% Proof A4
%-------------------------
\begin{lemma}[A4]\label{LemA4}
There exist constants $C_{\mathcal{R}},C_4\geq 1$ such that 
for all \(\hat{\mesh}\in\mathbb{T}\) there exists a subset \(\mathcal{R}(\mesh,\hat{\mesh})\subseteq \mesh\) with \(\mesh\backslash\hat{\mesh}\subseteq \mathcal{R}(\mesh,\hat{\mesh})\), \(|\mathcal{R}(\mesh,\hat{\mesh})|\leq C_{\mathcal{R}}|\mesh\backslash \hat{\mesh}|\) and

\begin{equation*}
\mathrm{d}[U(\hat{\mesh}),U(\mesh)]^2\leq C_4^2\sum\limits_{T\in \mathcal{R}(\mesh,\hat{\mesh})}\eta_T(\mesh;U(\mesh))^2.
\end{equation*}

\end{lemma}
\begin{proof}
Setting $\mathcal{R}(\mesh,\hat{\mesh}):=\mesh\backslash\hat{\mesh}$ we directly get the assertion from \eqref{discreteReliability}.
\end{proof}

%-----------------------
% Main Theorem
%-----------------------

\subsection{Convergence and quasi-optimal convergence}

First, we state that the discrete solutions $U(\mesh_\ell)$ converge to the solution $u$ with respect to the error measure $\mathrm{d}[\cdot,\cdot]$, i.e. $\mathrm{d}[u,U(\mesh_\ell)]\rightarrow 0$ as $\ell\rightarrow\infty$. This directly results from the following theorem, which is covered by \cite[Thm.4.1]{Axioms}.

\begin{theorem}\label{convergenceResult}
It holds for all $0< \theta \leq 1$ that there exists a $0<\rho<1$ such that

\begin{equation*}
\mathrm{d}[u,U(\meshl)]\lesssim \rho^{\ell/2}\:\eta(\mesh_0;U(\mesh_0))
\end{equation*}

\noindent for all $\ell\in\N_0$.
\end{theorem}
\begin{proof}
The assertion follows from Lemma \ref{LemA1}, Lemma \ref{LemA2}, Lemma \ref{LemA3} and Lemma \ref{LemA4} (resp. reliability \eqref{reliability}) together with \cite[Thm.4.1(i)]{Axioms}. 
\end{proof}

Second, we have quasi-optimal convergence, which is stated in the following theorem. To that end, let
\begin{equation*}
 \|(u, U(\cdot))\|_{\mathbb{B}_s}:=\sup_{N\in\mathbb{N}_0}\min_{\mesh\in\mathbb{T}(N)}(N+1)^s\:
\eta(\mesh;U(\mesh))
\end{equation*}

\noindent for $s>0$
where $\mathbb{T}(N):=\{\mesh\in\mathbb{T}(\mesh_0)\mid |\mesh|-|\mesh_0|\leq N\}$. 

\begin{theorem}\label{Main Theorem Optimality}
It holds for $0<\theta<(1+C_1^2C_4^2)^{-1}$ that

\begin{equation*}
\|(u, U(\cdot))\|_{\mathbb{B}_s}\lesssim \sup\limits_{\ell\in\N_0}\eta(\meshl;U(\meshl))\:\big(|\meshl|-|\mesh_0|+1\big)^{s}\lesssim \|(u, U(\cdot))\|_{\mathbb{B}_s}
\end{equation*}

\noindent for all \(s>0\).
\end{theorem}
\begin{proof}
The assertion follows from Lemma \ref{LemA1}, Lemma \ref{LemA2}, Lemma \ref{LemA3} and Lemma \ref{LemA4} together with \cite[Thm.4.1(ii)]{Axioms}. 
\end{proof}

%------------------------
% Comparison
%------------------------

\section{Some remarks on optimality proofs}
This section is devoted to comparing the optimality proof of this paper to the proof presented in \cite{ElastoAFEM} to highlight their similarities and differences, in particular with respect to convergence, the definition of optimality and optimal convergence. The analysis of these differences shows exemplary the applicability of the results in \cite{Axioms} to non-linear problems and emphasizes the advantages of the abstract approach provided in \cite{Axioms}.

\subsection{Convergence}
First, we examine the proof of the convergence of the algorithm. In \cite{ElastoAFEM} convergence is proven by showing the convergence of the weighted sum

\begin{equation}\label{WeightedSum}
\xi_{\ell}^2:=\eta(\meshl;U(\meshl))^2+~\beta\:\big(E(U(\meshl))-~E(u)\big).
\end{equation}

\noindent To that end, reliability of the error estimator is shown. Then (with the help of \cite[Lem.3]{ElastoAFEM} and \cite[Lem.4]{ElastoAFEM} as well as the Dörfler marking) the estimate 

\begin{align}\label{estimator reduction}
\eta(\mesh_{\ell+1};U(\mesh_{\ell+1})^2)\leq \rho\: \eta(\meshl;U(\meshl))^2+C\:\mathrm{d}[\mesh_{\ell+1};U(\mesh_{\ell+1}),U(\meshl)]^2,
\end{align}

\noindent is established, where $0<\rho<1$ and $C>0$. The estimator reduction given in \eqref{estimator reduction} is then used in combination with reliability \eqref{reliability} to obtain the convergence of the weighted sum $\xi_{\ell}$, which implies the convergence of the algorithm.\\
The convergence in this paper is given by Theorem \ref{convergenceResult}. As this theorem is a result of \cite[Thm.4.1]{Axioms} it is necessary to examine how that theorem is proven in \cite{Axioms} to make a meaningful comparison: The proof in \cite{Axioms} begins by establishing the reliability of the error estimator, similarly to the proof in \cite{ElastoAFEM}. However, in that case, reliability is a direct consequence of the "discrete reliability" axiom (A4). In the next step, the "stability" axiom (A1) and "reduction property" axiom (A2) are utilized as well as the Dörfler marking to show the estimator reduction \eqref{estimator reduction}.\\
Here, the axioms (A1) and (A2) in \cite{Axioms} have similar roles as \cite[Lem.3]{ElastoAFEM} and \cite[Lem.4]{ElastoAFEM}. We note that in order to prove the "reduction property" axiom (A2) we apply \cite[Lem.4]{ElastoAFEM} and only have to do minimal computations to prove the axiom. While \cite[Lem.3]{ElastoAFEM} has no direct application in the proof of axiom (A1), the proof does utilize the same proving method. Finally, the convergence of the error estimator is then shown by applying the error estimator reduction \eqref{estimator reduction} which allows the use of "general quasi-orthogonality" axiom (A3) in order to derive R-linear convergence.\\
Upon first examination, this last step in the proof of convergence differs from the one in \cite{ElastoAFEM}. However, we recall that the proof of "general quasi-orthogonality" axiom (A3) involves the equivalence of the R-linear convergence of the error estimator and axiom (A3). In that case, R-linear convergence is derived from the convergence of the weighted sum $\xi_{\ell}$ in \eqref{WeightedSum} and the equivalence of errors as stated in \eqref{ESigma}. Therefore, in essence the proof in \cite{ElastoAFEM} shows R-linear convergence, which implies axiom (A3), whereas in \cite{Axioms} and in this paper axiom (A3) implies R-linear convergence. This means that the last step in the convergence proof is done in a reverse order to the other.\\
In both instances, the proof may be summarized in the following three steps: First, verifying the reliability of the error estimator \eqref{reliability}, second, establishing estimator reduction \eqref{estimator reduction} and third, deriving convergence of the estimator. The main difference is found in the first step, where reliability of the estimator \eqref{reliability} is proven independently of "discrete reliability" axiom (A4) in \cite{ElastoAFEM}.

\subsection{Optimality}\label{Subsec Optimality}
We start with the differences in definitions of optimality of \cite{ElastoAFEM} and this paper and consider how they relate to one another. In \cite{ElastoAFEM} optimality is defined with the following semi norm

\begin{align*}
|(u,f,g)|_{\mathcal{A}_s}:=\sup\limits_{N\in\N}N^s\min\limits_{\mesh\in\meshset(N)}\Big(\operatorname{osc}^2(\mesh)+E(U(\mesh))-E(u)\Big)^{1/2}.
\end{align*}

\noindent Using \eqref{ESigma} we may replace the error of energies with \(\mathrm{d}[\mesh;u,U(\mesh)]^2\) as they are equivalent. The optimality estimate then reads

\begin{align*}
\big(|\meshl|-|\mesh_0|\big)^s\:\Big(\mathrm{d}[\meshl;u,U(\meshl)]^2+\operatorname{osc}^2(\meshl)\Big)^{1/2}\leq \overline{C}(s)\:|(u,f,g)|_{\mathcal{A}_s}.
\end{align*}

\noindent Applying \eqref{efficiency} yields the estimate

\begin{align*}
\big(|\meshl|-|\mesh_0|\big)^s\: \eta(\meshl;U(\meshl))&\leq \overline{C}(s)\:|(u,f,g)|_{\mathcal{A}_s}\\
&= \overline{C}(s)\:\sup\limits_{N\in\N}\min\limits_{\mesh\in\meshset(N)}N^s\:\Big(\mathrm{d}[\mesh;u,U(\mesh)]^2+\operatorname{osc}^2(\mesh)\Big)^{1/2}.
\end{align*}

\noindent Utilizing \eqref{osc estimate 1}, \eqref{osc estimate 2} as well as reliability \eqref{reliability} gives 

\begin{align*}
\big(|\meshl|-|\mesh_0|\big)^s\: \eta(\meshl;U(\meshl))&\leq \overline{C}(s)\:\sup\limits_{N\in\N}\min\limits_{\mesh\in\meshset(N)}N^s\:\eta(\mesh;U(\mesh)).
\end{align*}

\noindent If the supremum is taken over \(\ell\in \N\), then this is the optimality estimate of Theorem \ref{Main Theorem Optimality}.\\
Conversely, from Theorem \ref{Main Theorem Optimality} we derive the estimate

\begin{align}\label{Optimality Equvalence 2}
\sup\limits_{\ell\in\N_0}3\:\eta(\meshl;U(\meshl))\big(|\meshl|-|\mesh_0|+1\big)^{s}\leq C\:\norm{\eta(\cdot),U(\cdot)}_{\mathbb{B}_s}.
\end{align}

\noindent A lower bound of the left-hand side of \eqref{Optimality Equvalence 2} results from \eqref{osc estimate 1}, \eqref{osc estimate 2} and reliability \eqref{reliability}

\begin{align}\label{Optimality Equivalence 3}
\big(|\meshl|-|\mesh_0|\big)^{s}\:\Big(\mathrm{d}[\meshl;u,U(\meshl)]^2+\operatorname{osc}^2(\meshl)\Big)^{1/2}\leq\sup\limits_{\ell\in\N}3\:\eta(\meshl;U(\meshl))\:\big(|\meshl|-|\mesh_0|\big)^{s}.
\end{align}

\noindent Estimating the right-hand side of \eqref{Optimality Equvalence 2} with the efficiency estimate \eqref{efficiency}, gives

\begin{equation}\label{Optimality Equivalence 4}
\begin{split}
C\:\norm{\eta(\cdot),U(\cdot)}_{\mathbb{B}_s}&\lesssim \sup\limits_{N\in\N}\min\limits_{\mesh\in\meshset(N)}N^s\Big(\mathrm{d}[\mesh;u,U(\mesh]^2+\operatorname{osc}^2(\mesh)\Big)^{1/2}\\
&\simeq |(u,f,g)|_{\mathcal{A}_s}.
\end{split}
\end{equation}

\noindent Combining \eqref{Optimality Equivalence 3} with \eqref{Optimality Equivalence 4} yields

\begin{align*}
\big(|\meshl|-|\mesh_0|\big)^{s}\Big(\mathrm{d}[\meshl;u,U(\meshl)]^2+\operatorname{osc}^2(\meshl)\Big)^{1/2}\lesssim|(u,f,g)|_{\mathcal{A}_s}.
\end{align*}

\noindent Therefore, we see that both definitions of optimality are equivalent and thus implies that if the algorithm convergences optimally in the sense of \cite{Axioms} it also converges optimally in the sense of \cite{ElastoAFEM} and vice versa.

\subsection{Optimal convergence}

The proof of optimality in \cite[Thm.7]{ElastoAFEM} begins by establishing the existence of a triangulation \(\overline{\meshl}\in\meshset(N_{\ell})\) such that the estimates

\begin{align}
E(U(\overline{\meshl}))-E(u)+\operatorname{osc}^2(\overline{\meshl})&\leq (\tau\:\xi_{\ell})^2\label{NEstimate0}\\
|(u,f,g)|_{\mathcal{A}_s}&\leq \tau\:\xi_{\ell}\:N_{\ell}^s\label{NEstimate}
\end{align}

\noindent hold with \(N_{\ell}\in\N\) being the smallest number such that \eqref{NEstimate} is satisfied. These estimates are shown only by exploiting the definition of the approximation class \(\mathcal{A}_s\) and the minimality of \(N_{\ell}\).\\
The proof of optimality presented in this paper (Theorem~\ref{Main Theorem Optimality}) is based on \cite[Thm.4.1]{Axioms}. Therefore, we take a closer look at this theorem. Comparing \eqref{NEstimate0} and \eqref{NEstimate} to the proof of \cite[Lem.4.14]{Axioms} (which is used in \cite[Thm.4.1]{Axioms}) we observe certain similarities. The existence of an \(N\in\N\) and a triangulation \(\mesh_{\varepsilon}\in\meshset(N)\) is shown such that

\begin{align}\label{Axiom epsilon optimal}
\eta(\mesh_{\varepsilon};U(\mesh_{\varepsilon}))=\min\limits_{\mesh\in\meshset(N)}\eta(\mesh;U(\mesh))\leq (N+1)^{-s}\:\norm{(\eta(\cdot),U(\cdot))}_{\mathbb{B}_s}\leq \lambda\:\eta(\meshl;U(\meshl)).
\end{align}

\noindent Here, \(N\in \N_0\) is chosen as the smallest number satisfying

\begin{align}\label{Axiom epsilon optimal 2}
\norm{(\eta(\cdot),U(\cdot))}_{\mathbb{B}_s}\leq \lambda\:\eta(\meshl;U(\meshl))\:(N+1)^{s}.
\end{align}

\noindent Obviously, $\xi_{\ell}$ and $\eta(\meshl;U(\meshl))$ take on the same roles in both estimates. The difference of these two terms might be explained by the oscillation terms given in the definition of optimality in \cite{ElastoAFEM}.\\
The proof of optimality in \cite[Thm.7]{ElastoAFEM} continues by showing that the estimate 

\begin{equation}\label{MeshEstimeatePaper}
|\meshl\backslash (\meshl\oplus\overline{\meshl})|\leq N_{\ell}\leq 2\:|(u,f,g)|_{\mathcal{A}_s}^{1/s}\:(\tau\:\xi_{\ell})^{-1/s}
\end{equation}

\noindent holds for the triangulation $\overline{\mesh}_{\ell}$. This is done by exploiting the properties of the common coarsest refinement $\meshl\oplus\overline{\meshl}$ and the minimality of $N_{\ell}$. 
Similarly to the estimate \eqref{MeshEstimeatePaper}, it is proven in \cite{Axioms} that

\begin{equation*}
|\mathcal{R}(\meshl;\mesh_{\varepsilon}\oplus\meshl)|\simeq|(\meshl\backslash \mesh_{\varepsilon}\oplus\meshl)|\leq N \lesssim \norm{(\eta(\cdot),U(\cdot))}_{\mathbb{B}_s}^{1/s}\big(\lambda\:\eta(\meshl;U(\meshl))\big)^{-1/s},\\
\end{equation*}

\noindent where $\mathcal{R}(\meshl;\mesh_{\varepsilon}\oplus\meshl)$ is the set described by "discrete reliability" axiom (A4). This set is shown to satisfy the Dörfler marking in \cite{Axioms}, which requires the application of the "stability" axiom (A1) and the "discrete reliability" axiom (A4), as well as the quasi-monotonicity  of the error estimator. The latter is implied by the axioms (A1), (A2) and (A4).\\
In \cite{ElastoAFEM} the set $\meshl\backslash (\meshl\oplus\overline{\meshl})$ is proven to satisfy the Dörfler marking as well. However, this proof requires the efficiency of the error estimator. The proof of Theorem~\ref{Main Theorem Optimality} in this paper does not need the efficiency of the error estimator at any point, which is thus far the main difference of the two proof methodologies.\\\\
As $\meshl\backslash (\meshl\oplus\overline{\meshl})$ satisfies the Dörfler marking criterion, it follows that

\begin{align}\label{Optimality alt 1}
|\mathcal{M}_{\ell}|\leq |\meshl\backslash (\meshl\oplus\overline{\meshl})|\leq N_{\ell}\leq 2\:|(u,f,g)|_{\mathcal{A}_s}^{1/s}(\tau\:\xi_{\ell})^{-1/s}.
\end{align}

\noindent This together with the optimality of the mesh closure and the convergence of the weighted sum $\xi_{\ell}$ \eqref{WeightedSum} gives the following chain of inequalities

\begin{equation}\label{Optimality alt 2}
\begin{split}
|\meshl|-|\mesh_0|&\lesssim \sum\limits_{k=0}^{\ell-1}|\mathcal{M}_{\ell}|\lesssim |(u,f,g)|_{\mathcal{A}_s}^{1/s}\:\tau^{-1/s}\sum\limits_{k=0}^{\ell-1}\xi_k^{-1/s}\\
&\lesssim |(u,f,g)|_{\mathcal{A}_s}^{1/s}\:\tau^{-1/s}\:\xi_{\ell}^{-1/s}.
\end{split}
\end{equation}

\noindent In \cite{Axioms} the Dörfler marking criterion leads to

\begin{align}\label{Optimality Axiom 1}
|\mathcal{M}_{\ell}|\lesssim |\mathcal{R}(\meshl;\mesh_{\varepsilon}\oplus\meshl) |\lesssim \norm{(\eta(\cdot),U(\cdot))}_{\mathbb{B}_s}^{1/s}\:\eta(\meshl;U(\meshl))^{-1/s}, 
\end{align}

\noindent which then combined with the optimality of the mesh closure implies

\begin{equation}\label{Optimality Axiom 2}
\begin{split}
|\meshl|-|\mesh_0|+1 &\lesssim \sum\limits_{j=0}^{\ell-1} |\mathcal{M}_{\ell}| \lesssim \norm{(\eta(\cdot),U(\cdot))}_{\mathbb{B}_s}^{1/s} \sum\limits_{j=0}^{\ell-1} \eta(\mesh_j;U(\mesh_j))^{-1/s}\\
&\lesssim \norm{(\eta(\cdot),U(\cdot))}_{\mathbb{B}_s}^{1/s}\:\eta(\meshl;U(\meshl))^{-1/s}.
\end{split}
\end{equation}

\noindent The last estimate is due to the inverse summability of the error estimator, which is shown to be equivalent to the R-linear convergence in \cite{Axioms}.
It is easily seen that the step \eqref{Optimality Axiom 1} is identical to the statement \eqref{Optimality alt 1}. The last estimations in \eqref{Optimality alt 2} and \eqref{Optimality Axiom 2} respectively are also based on the same arguments. The inverse summability of the estimator is a consequence of "general quasi-orthogonality" axiom (A3), while the estimation in \eqref{Optimality alt 2} is based on R-linear convergence and inverse summability for the weighted sum $\xi_{\ell}$.\\
Once \eqref{Optimality Axiom 2} and \eqref{Optimality alt 2} are established, the proofs of optimality are finished with some simple computations. Again, we see that \(\xi_{\ell}\) \eqref{WeightedSum} takes on a similar role as
\(\eta(\meshl;U(\meshl))\). Indeed, the meshes \(\overline{\meshl}\) and \(\mesh_{\varepsilon}\) take on the same role in the two respective proofs.

\subsection{Conclusion}
From this analysis we can infer that while it is not immediately obvious in \cite{ElastoAFEM}, the proof of optimality and convergence implicitly used the axioms, as we can see with \cite[Lem.3]{ElastoAFEM}, \cite[Lem.4]{ElastoAFEM} and more directly with \cite[Thm.3]{ElastoAFEM}. The one main difference is the necessity of the efficiency of the error estimator in \cite{ElastoAFEM}, which is completely absent in the optimality proof of \cite[Thm.4.1]{Axioms} and - by extension - the proof of optimality presented in this paper. The definition of the approximation class \(\mathcal{A}_s\), namely the inclusion of the oscillation term \(\operatorname{osc}^2(\mesh)\), requires the efficiency of the error estimator. Furthermore, the efficiency is needed to prove that the definition of optimality in \cite{ElastoAFEM} and \cite{Axioms} are equivalent. The lack of the use of efficiency is emphasized in \cite{Axioms}.\\
Apart from this difference, the way optimality is established in \cite{ElastoAFEM} and \cite{Axioms} show remarkable similarities, despite the initial impression that the methodologies are rather different. However, the use of the results of \cite{Axioms} to establish convergence and optimal convergence, as presented in this paper, significantly streamlines the proof. It removes the need for an involved proof of optimality, see proof of \cite[Thm.7]{ElastoAFEM}. Instead, after establishing optimality through \cite[Thm.4.1]{Axioms}, only the equivalence of the definitions of optimality has to be verified, which is shown in Subsection \ref{Subsec Optimality} through some simple computations. We emphasize, however, that while the application of the axioms simplifies the proof of optimal convergence, verifying the axioms still requires effort. The use of results of \cite{ElastoAFEM} allows us to show that each axiom is satisfied, but without such results the verification of the axioms would require considerably more work.

\begin{funding}
  The authors gratefully acknowledge the support by the Bundesministerium für Bildung, Wissenschaft und Forschung (BMBWF) under the Sparkling Science project SPA 01-080 "MAJA- Mathematische Algorithmen für Jedermann Analysiert".
\end{funding}

\end{document}